\documentclass[12pt]{article}
\usepackage{latexsym}
\usepackage{epsfig}
\usepackage{amsmath,amsthm,amssymb,enumerate}
\usepackage[active]{srcltx}

\usepackage{hyperref}

\usepackage{color}

\definecolor{plum}{rgb}{0.62, 0.0, 0.77}

\def\plum{\color{plum}}
\def\wE{\wh{E}}
\def\wV{\wh{V}}
\def\wK{\wh{K}}
\def\chW{\wh{\cW}}
\def\wT{\wh{T}}
\def\cU{\mathcal U}

\newcounter{rot}%\addtocounter{rot}{1}, \therot

\def\cN{{\cal N}}

\def\nn{\nonumber}
\def\a{\alpha} \def\b{\beta} \def\d{\delta} 
\def\e{\varepsilon} \def\f{\phi} \def\F{{\Phi}}  \def\g{\gamma}
  \def\k{\kappa}
 \def\th{\theta}    \def\l{\lambda}
\def\La{\Lambda} \def\m{\mu} \def\n{\nu} \def\p{\pi}
\def\r{\rho}  \def\s{\sigma} 
\def\t{\tau} \def\om{\omega}  \def\U{\Upsilon}

\newtheorem{theorem}{Theorem}
\newtheorem{lemma}[theorem]{Lemma}

\newtheorem{Remark}{Remark}
\newtheorem{claim}{Claim}

\def\cW{{\mathcal W}}
\def\cX{{\mathcal X}}

\def\tE{\E}
\newcommand{\wh}[1]{\widehat{#1}}

\newcommand{\rdup}[1]{{\left\lceil #1 \right\rceil }}
\newcommand{\rdown}[1]{{\left\lfloor #1\right \rfloor}}

\newcommand{\brac}[1]{\left(#1\right)}

\newcommand{\bfrac}[2]{\left(\frac{#1}{#2}\right)}

\newcommand{\set}[1]{\left\{#1\right\}}

\def\E{\mbox{{\bf E}}}
\def\Var{\mbox{{\bf Var}}}
\def\Pr{\mbox{{\bf Pr}}}

\newcommand{\ignore}[1]{}

\newcommand{\cA}{{\cal A}}

\newcommand{\card}[1]{\left|#1\right|}

\newcommand{\beq}[2]{\begin{equation}\label{#1}#2\end{equation}}

\def\eff{\mathrm{eff}}

\def\cN{{\cal N}}

\newcommand{\mults}[1]{\begin{multline*}#1\end{multline*}}
\newcommand{\mult}[2]{\begin{multline}\label{#1}#2\end{multline}}

\def\wR{\widehat{R}}
\def\tR{\tilde{R}}
\def\wheta{\widehat{\eta}}
\def\dist{\text{dist}}

\usepackage{tikz}
\usetikzlibrary{arrows}
\usetikzlibrary{decorations}
\usetikzlibrary{shapes.misc}

\newcommand{\pics}[2]{
\begin{tikzpicture}[scale=#1]
#2
\end{tikzpicture}}

\author{Alan Frieze\thanks{Research supported in part by NSF grant DMS1661063
}\qquad Wesley Pegden\thanks{Research supported in part by NSF grant DMS1363136}\qquad Tomasz Tkocz\thanks{Research supported in part by NSF grant DMS1955175} \\Department of Mathematical Sciences\\Carnegie Mellon University\\Pittsburgh PA15217\\U.S.A.}

\begin{document}
\title{On the cover time of the emerging giant}
\maketitle
\begin{abstract}
Let $p=\frac{1+\e}{n}$. It is known that if $N=\e^3n\to\infty$ then with high probability (w.h.p.) $G_{n,p}$ has a unique giant largest component. We show that if in addition, $\e=\e(n)\to 0$ then w.h.p. the cover time of $G_{n,p}$ is asymptotic to $n\log^2N$; previously Barlow, Ding, Nachmias and Peres had shown this up to constant multiplicative factors.
\end{abstract}
\section{Introduction}
Let $G=(V,E)$ be a connected graph with vertex set  $V$ of size $n$ and an edge set $E$. In a simple random walk $\cW$ on a graph $G$, at each step, a particle moves from its current vertex to a randomly chosen neighbor. For $v\in V$, let $C_v$ be the expected time taken for a simple random walk starting at $v$ to visit every vertex of $G$. The {\em vertex cover time} $C_G$ of $G$ is defined as $C_G=\max_{v\in V}C_v$. The (vertex) cover time of connected graphs has been extensively studied. It was shown by  Feige \cite{Feige1}, \cite{Feige2}, that for any connected graph $G$, the cover time satisfies $(1-o(1))n\log n\leq C_G\leq (1+o(1))\frac{4}{27}n^3.$ 

In a series of papers, Cooper and Frieze have asymptotically established the cover time in a variety of random graph models. The following theorem lists some of the main results. (Here $A_n\approx B_n$ if $A_n=(1+o(1))B_n$ as $n\to\infty$.)
\begin{theorem}\label{thCF}
The following asymptotic estimates for the cover time hold with high probability (w.h.p.):
\begin{description}
\item[\cite{CF1}] If $G=G_{n,p}$ with $p=\frac{c\log n}{n}$, $c>1$, then $C_G\approx \f(c)n\log n$ where $\f(c)=c\log\bfrac{c}{c-1}$.
\item[\cite{CF2}] If $G=G_{n,r}$ with $r=O(1)$, a random $r$-regular graph, then $C_G\approx \frac{r-1}{r-2}n\log n$.
\item[\cite{CF3}] Let $G=G_{n,d,r}$ with $d\geq 3$ and  $r=\bfrac{c\log n}{\Upsilon_d n}^{1/d}$ be the random geometric graph on $n$ vertices in dimension $d$\footnote{Here $\Upsilon_d$ is the volume of the Euclidean ball of radius one in $\mathbb{R}^d$. The random geometric graph $G = G_{n,d,r}$ is defined as follows: we choose $n$ points independently uniformly at random from $[0,1]^d$ to be the vertices of $G$ and two points are joined by an edge if and only if they are at most distance $r$-apart.}. Then $C_G\approx \f(c)n\log n$.
\item[\cite{CF4}] If $D=D_{n,p}$ (the random digraph counterpart of $G_{n,p}$), then $C_D\approx \f(c)n\log n$.
\end{description}
\end{theorem}
Cooper and Frieze \cite{CFgiant} also established the cover time of the giant component $C_1$ of the random graph $G_{n,p}$ with $p=c/n$, where $c>1$ is a constant. They showed in this setting that w.h.p. the cover time $C_{C_1}$ satisfies
$$C_{C_1}\approx \frac{cx(2-x)}{4(cx-\ln c)} n (\ln n)^2,$$
where $x$ denotes the solution in $(0,1)$ of $x=1-e^{-cx}$.

This raises the question as to what happens in $G_{n,p}$ if $p=(1+\e)/n,\,\e>0$ and we allow $\e\to 0$. It is known that a unique giant component emerges w.h.p. only when $\e^3n\to\infty$. Barlow, Ding, Nachmias and Peres \cite{BDNP} showed that w.h.p. 
\beq{Barl}{
C_{C_1}=\Theta(n\log^2(\e^3n)).
} 
Cooper, Frieze and Lubetzky \cite{CFL} showed that if $C_1^{(2)}$ denotes the 2-core of the giant component $C_1$ of $G_{n,p}$ ($C_1$ stripped of its attached trees), then, in this range of $p$, w.h.p.  $C_{C_1^{(2)}}\approx \frac{1}{4}\e n\log^2(\e^3n)$, but they were not able to determine the cover time of the giant $C_1$ asymptotically. We do this in the current paper, confirming their conjecture.

We prove the following theorem:
\begin{theorem}\label{th1}
Let $p = \frac{1+\e}{n}$ with $\e = \e(n) > 0$, $\e \to 0$ such that $\e^3n\to\infty$. Let $C_1$ be the giant component of $G_{n,p}$. Then w.h.p. 
$$C_{C_1}\approx n\log^2(\e^3n).$$
\end{theorem}
%This confirms a conjecture from \cite{CFL}, where it was shown that $C_{C_1^{(2)}} \approx \frac{\e}{4}n\log^2(\e^3n)$ ($C_1^{(2)}$ is the $2$-core of $C_1$, that is $C_1$ stripped of its attached trees). 
Our proof is very different from the proof in \cite{CFL}. We will use the notion of a Gaussian Free Field (GFF). This was used in the breakthrough paper of Ding, Lee and Peres \cite{DLP} that describes a {\em deterministic} algorithm for approximating $C_G$ to within a constant factor. This was later refined by Ding \cite{Ding} and by Zhai \cite{Z}. It is the latter paper that we will use. In the next section, we will describe the tools needed for our proof. Then in Section \ref{prf} we will use these tools to prove Theorem \ref{th1}.
\section{Tools}
\subsection{Gaussian Free Field}

\textbf{Definition 1.}
For our purposes, given a graph $G=(V,E)$, a GFF is a {\plum random} vector $(\eta_v,v\in V)$ whose joint distribution is Gaussian
with
\begin{enumerate}[(i)]
\item $\E(\eta_v)=0$ for all $v\in V$.
\item $\eta_{v_0}=0$ for some fixed vertex $v_0\in V$.
\item $\E((\eta_v-\eta_w)^2) =R_{\eff}(v,w)$ for all $v,w\in V$.
\end{enumerate} 
Note that in particular, $\text{Var}(\eta_v) = \E (\eta_v^2) = R_{\eff}(v,v_0)$. (Here $R_{\eff}$ is the effective resistance between $v$ and $w$, when $G$ is treated as an electrical network where each edge is a resistor of resistance one. See Doyle and Snell \cite{DS} or Lewin, Peres and Wilmer \cite{LPW} for nice discussions of this notion). As its name suggests, $R_{\eff}$ is most naturally defined in terms of electrical networks. For us the following mathematical definition will suffice: for a graph $G=(V,E)$ and vertices $v,w\in V$, we use the {\em commute time identity} to define
\beq{Reffdef}{
R_{\eff}(v,w)=\frac{\t(v,w)+\t(w,v)}{2|E|},
}
where $\t(v,w)$ is the expected time for a simple random walk starting at $v$ to reach $w$.

Note that, as suggested by the electrical analog, we have
\beq{Reffdist}{
R_{\eff}(v,w)\leq \dist(v,w).
}
This is a simple consequence of Rayleigh's Monotonicity Law (delete all edges except for a shortest path from $v$ to $w$), see \cite{DS}.

In the continuous setting, the Gaussian free field generalizes Brownian motion (or the Brownian bridge) and can be seen as a model of a random surface.  In the discrete setting, the Gaussian Free Field can be seen as generalizing Brownian motion on a line to an analog of Brownian motion on the topology of the graph.  In particular, if $G$ is a path with $t$ edges, and the fixed vertex $v_0$ is an endpoint of the path, then the normals $\eta_v$ in the GFF for the path can be generated in terms of Brownian motion $W(t)$, by setting $\eta_v$ to be $W(\dist(v,v_0))$.  

The important thing for the present paper is a remarkable connection between the Gaussian Free Field on a graph and its cover time.  Let us define
$$M=\E(\max_{v\in V}\eta_v).$$
Ding, Lee and Peres \cite{DLP} proved that there are universal constants $c_1,c_2$ such that
\beq{DLPth}{
c_1|E|M^2\leq C_G\leq c_2|E|M^2.
}
Next let $R=\max_{v,w\in V}R_{\eff}(v,w)$. Zhai \cite{Z} proved the following theorem:
\begin{theorem}[Zhai]\label{Zhai-theorem}
Let $G = (V,E)$ be a finite undirected graph with a specified vertex $v_0 \in V$. There are universal positive constants $c_1, c_2$ such that if we let $\tau_{cov}$ be the first time that all the vertices in $V$ have been visited at least once for the walk on $G$ started at $v_0$, we have
\beq{Zth}{
\Pr\left(\Big|\tau_{cov}-|E|M^2\Big|\geq |E|(\sqrt{\l R}\cdot M+\l R)\right)\leq c_1e^{-c_2\l}
}
for any $\l\geq c_1$. 
\end{theorem}
Setting $X = \frac{\tau_{cov}}{|E|M^2}$, this gives after crude estimates
\[
|\E X - 1| \leq \E|X-1| = \int_0^\infty \Pr(|X-1| > t) dt \leq C\left(\sqrt{\frac{R}{M^2}} + \frac{R}{M^2}\right)
\]
for a universal constant $C$. Note that $R$ and $M$ do not depend on $v_0$ (for $M$, observe that for any fixed vertex $w$, $\E[\max_{v\in V}\eta_v] = \E[(\max_{v\in V}(\eta_v-\eta_w)) + \eta_w] = \E[\max_{v\in V}(\eta_v-\eta_w)]$, since the Gaussians have mean $0$, see also Remark 1.3 in \cite{Z}). After taking the maximum over $v_0$ we thus get that $C_G = \max_{v_0} \E\tau_{cov}$ satisfies
\begin{equation}\label{eq:C_G}
C_G = |E|M^2\left(1 + O\left(\sqrt{\frac{R}{M^2}} + \frac{R}{M^2}\right)\right).
\end{equation}
Now, as we will see in the next section, the number of edges in the emerging giant is given by the following theorem:
\begin{theorem}\label{numberedges} Let $G = G_{n,p}$ be as in Theorem \ref{th1}. Then
\beq{edges}{
|E(C_1)| \approx2\e n\quad w.h.p.
}
\end{theorem}
This follows from the work in \cite{DKLP} as we will see in Section \ref{struct}. 

Our main contribution is the following theorem:
\begin{theorem}\label{Mexp} 
Let $G = G_{n,p}$ be as in Theorem \ref{th1} and let $M$ the the expected maximum of a GFF on $G$ as defined above. Then
\beq{EG1}{
M\approx \frac{\log (\e^3n)}{(2\e)^{1/2}} \quad w.h.p.
}
\end{theorem}
This immediately implies Theorem \ref{th1} as follows:
\begin{proof}[Proof of Theorem \ref{th1}]
In view of \eqref{eq:C_G} obtained from Theorem \ref{Zhai-theorem}, Theorem \ref{Mexp} implies Theorem \ref{th1} if we can show that w.h.p. $R=o(M^2)$. Now, we know from \eqref{Barl}, \eqref{DLPth} and \eqref{edges} (or from Theorem~\ref{Mexp}) that w.h.p. $M=\Omega(\e^{-1/2}\log(\e^3n))$. Therefore to prove that $R=o(M^2)$ it will be sufficient to prove
\beq{smallR}{
R=O\bfrac{\log(\e^3n)}{\e}.
}
This can be verified as follows: first we observe that the effective resistance between two vertices of a graph $G$ is always bounded above by the diameter of $G$, see \eqref{Reffdist}. Second, it was proved in \cite{DKLP2} that w.h.p. the diameter of $G_{n,p}$ is asymptotically equal to $\frac{3\log(\e^3n)}{\e}$ and so \eqref{smallR} follows immediately.
\end{proof}
\subsection{Structure of the emerging giant}\label{struct}
Ding, Kim, Lubetzky and Peres \cite{DKLP} describe the following construction of a random graph, which we denote by $H$. Let $0<\m<1$ satisfy $\m e^{-\m}=(1+\e)e^{-(1+\e)}$. Let $\cN(\m,\s^2)$ denote the normal distribution with mean $\m$ and variance $\s^2$.

{\sc giantconstruction} 
\begin{enumerate}[Step 1.]
\item Let $\Lambda\sim \cN\brac{1+\e-\m,\frac{1}{\e n}}$ and assign i.i.d. variables $D_u\sim \text{Poisson}(\Lambda)$ ($u\in[n]$) to the vertices, conditioned that $\sum D_u1_{D_u\geq 3}$ is even. (While $\Lambda$ can be negative, we show in \eqref{deflambda} below that it is positive w.h.p.)\\
Let $N_k=|\set{u:D_u=k}|$ and $N_{\geq 3}=\sum_{k\geq 3}N_k$. Select a random graph $K_1$ on $N_{\geq 3}$ vertices, uniformly among all graphs with $N_k$ vertices of degree $k$ for all $k\geq 3$.
\item Replace each edge $e\in E(K_1)$ by a path $P_e$ of length $\text{Geom}(1-\m)$ to create $K_2$.  (Hereafter, $K_1$ denotes the graph from Step 1 whose vertices are the subset of vertices of $H$ consisting of these original vertices of degree $\geq 3$ and $K_2\supseteq K_1$ denotes the graph created by the end of this step.)
\item Attach an independent Poisson($\m$)-Galton-Watson tree with root $v$ to each vertex $v$ of $K_2$.
\end{enumerate}
The main result of \cite{DKLP} is the following theorem:
\begin{theorem}\label{themg}
Let $\e \to 0$ such that $\e^3n\to\infty$. For any graph property $\cA$, $\Pr(H\in\cA)\to0$ implies that $\Pr(C_1\in \cA)\to0$.
\end{theorem}
We will work with this construction for the remainder of the manuscript.
For our application of the Gaussian free field, we make the convenient choice that $v_0$ is a vertex in $K_1$. 

\begin{proof}[Proof of Theorem \ref{numberedges}.]
Let $H$ be the graph constructed in Steps 1-3. In view of Theorem \ref{themg}, in order to show $|E(C_1)| \approx 2\e n$, we show $|E(H)| \approx 2\e n$. We observe that
\beq{valmu}{
1-\m-\e\in [0,\e^2].
}
Recall from Step 1 that $\Lambda\sim \cN\brac{1+\e-\m,\frac{1}{\e n}}$. Applying the Chebyshev inequality we see that for any $\th>0$, we have
$$\Pr\brac{|\Lambda-\E(\Lambda)|\geq \th}\leq \frac{1}{\th^2\e n}.$$
Putting $\th=n^{-1/3}$, we see that $\th^2\e n = \e n^{1/3} \to \infty$, so
\beq{deflambda}{
\Lambda = \E \Lambda + O(n^{-1/3}) = 2\e+O(n^{-1/3}+\e^2),\qquad w.h.p.
}
The restriction $\sum D_u1_{D_u\geq 3}$ is even will be satisfied with constant probability and then we see that w.h.p. 
\beq{sizeK1}{
N_{\geq 3}\approx \frac{4}{3}\e^3n \text{ and almost all vertices of $K_1$ have degree three.}
}
Therefore, w.h.p.,
\beq{EK1}{
 |E(K_1)| \approx \frac{3}{2}\frac{4}{3}\e^3n = 2\e^3n
}
The expected length of each path constructed by Step 2 is asymptotically equal to $1/(1-\m)\approx 1/\e$. The path lengths are independent with geometric distributions (which have exponential tails) and so their sum is concentrated around their mean (by virtue of, e.g. Bernstein's inequality) which is asymptotically equal to $|E(K_1)|\frac{1}{\e}\approx 2\e^2n$. Thus, w.h.p., $|E(K_2)| \approx 2\e^2n$. 
Note also that in $K_2$, w.h.p.,~there is no path longer than $\frac{2}{\e}\log N_{\geq 3}$.

Furthermore, the expected size of each tree in Step 3 is also asymptotically equal to $1/\e$. These trees are independently constructed whose sizes also have exponentially decaying tails and so the total number of edges is concentrated around its mean which is asymptotically equal to $|E(K_2)|\frac{1}{\e} \approx 2\e n$. Thus, w.h.p. $|E(H)| \approx 2\e n$, which proves Theorem \ref{numberedges}.
\end{proof}
Let 
\[
N=\e^3n \text{ and let $\k$ denote the smallest power of 2 which is at least $1/\e$}.
\]

\begin{lemma}\label{longpath}
W.h.p. $|P_e|\leq \frac{2\log N}{\e}$ for all paths $P_e$ created in Step 2.
\end{lemma}
\begin{proof}
\[
\Pr\brac{|P_e|\geq \frac{2\log N}{\e}}\leq (1-\e(1-\e))^{2\log N/\e}\leq N^{-(2-o(1))}.
\]
The result now follows from \eqref{EK1} and the Markov inequality.
\end{proof}
\subsubsection{Galton-Watson Trees}
A key parameter for us will be the probability that a Galton-Watson tree with Poisson($\mu$) offspring distribution survives for at least $k$ levels.  The following Lemma was proved by Ding, Kim, Lubetzky and Peres (see Lemma 4.2 in \cite{DKLP2}).
\begin{lemma}\label{lemdepth}
Let $0<\m<1$ and $\e > 0$ satisfy $\m e^{-\m}=(1+\e)e^{-(1+\e)}$. Let $T$ be a Poisson($\m$)-Galton-Watson tree. Let $L_k$ denote the $k$-th level of $T$. Then there exist absolute constants $c_1<c_2$ such that for any $k\geq 1/\e$ we have
  \[
 c_1(\e\exp\set{-k(\e+c_1\e^2)})\leq  \Pr\left(L_k\neq\emptyset\right)\leq c_2(\e\exp\set{-k(\e-c_2\e^2)}).
  \]
\end{lemma}

Their proof also easily gives the following result.
\begin{lemma}\label{lemsmalldepth}
For $k<1/\e$ we have
  \[
  \Pr\left(L_k\neq\emptyset\right)<\frac{10}{k}.
  \]
\end{lemma} 
We shall need the following result about trees attached in Step 3. Here and throughout the remainder of the paper,
\[
N = \e^3n.
\]
\begin{lemma}\label{xlem1}\
Consider the construction of the graph $H$ from Steps 1-3. Let $0 < \gamma < 1$. Let $\mathcal{T}$ be the set of trees attached in Step 3 of {\sc giantconstruction}. Then, w.h.p. (referring to the entire construction, not just Step 3), we have:
\begin{enumerate}[(a)]
\item
\beq{BinTrees}{
\text{There are between $\frac12c_1N^{1-\g+O(\e)}$ and $2c_2N^{1-\g+O(\e)}$}\text{ trees in $\mathcal{T}$ of depth at least }\g\e^{-1}\log N.
}
\item 
\beq{notrees}{
\text{There are no trees in $\mathcal{T}$ of depth exceeding }\frac{2\log N}{\e}.
}
In fact the probability of the event in \eqref{notrees} is $1-O(N^{-(1-o(1))})$.
\end{enumerate}
Here $c_1, c_2 > 0$ are the universal constants from Lemma \ref{lemdepth}.
\end{lemma}
\begin{proof}
(a) Let $p_\g$ denote $\Pr(L_k\neq 0)$ for $k_\g=\lfloor\g\e^{-1}\log N\rfloor$,  $\g > 0$. Conditioning on the results of Step 1 and Step 2, the number $\n_\g$ of trees created in Step 3 of depth at least $k$ is a binomial with number of trials $|V(K_2)|$ and probability of success $p_\g$. Recall $|V(K_2)| \approx (1+o(1))\frac{4N}{\e}$. It follows from Lemma \ref{lemdepth} that 
\[(1+o(1))\frac{4N}{3}\cdot \frac{1}{\e}\cdot c_1\e\exp\set{-(\g+O(\e))\log N)}=\frac{4c_1}{3}N^{1-\g+O(\e)}\leq  \E(\n_\g)\leq \frac{4c_2}{3}N^{1-\g+O(\e)}.
\] 
Since $1 - \g > 0$ and $\e \to 0$, note that eventually $1-\g + O(\e) > \delta_0$ for some positive universal constant $\delta_0$, so $N^{1-\g+O(\e)} \to \infty$.

Thus conditional on the results of Step 1 and Step 2, $\n_\g$ is distributed as a binomial with mean going to infinity and so we have that if $0<\g<1$ then the Chernoff bounds imply \eqref{BinTrees}.

(b) It follows from Lemma \ref{lemdepth} that the probability that any fixed tree has depth at least $2\e^{-1}\log N$ is $O(\e N^{-2-o(1)})$. There are w.h.p. $O(\e^2n)$ trees and so the expected number of trees with this or greater depth is $O(\e^2n\times \e N^{-2-o(1)})=O(N^{-1-o(1)})$. The result now follows from the Markov inequality.
\end{proof}
\subsection{Normal Properties}\label{normal}
In this section we describe several properties of the normal distribution that we will use in our proof.

First suppose that $g_1,g_2,\ldots,g_s$ are independent copies of $\cN(0,1)$. Then if $G_s=\max_{i=1,\ldots,s} g_i$,
\beq{maxnorm}{
\E\brac{G_s}= \sqrt{2\log s} - \frac{\log\log s + \log(4\pi) - 2\gamma}{\sqrt{8\log s}} + O\bfrac{1}{\log s}
}
where $\gamma = 0.577\ldots$ is the Euler–-Mascheroni constant. For a proof see Cram\'er \cite{Cram}. 

Next suppose that $(X_i)_{1 \leq i \leq s}$ and $(Y_i)_{1 \leq i \leq s}$ are two centered Gaussian vectors in $\mathbb{R}^s$ such that $\E(X_i-X_j)^2 \leq \E(Y_i-Y_j)^2$ for all $1\leq i,j\leq s$. Then,
\beq{maxcomp}{
\E(\max\set{X_i:i=1,2,\ldots,s})\leq \E(\max\set{Y_i:i=1,2,\ldots,s})
}
(sometimes refered to as Slepian's lemma). See Fernique \cite{Fern} (Theorem 2.1.2 and Corollary 2.1.3).  Finally we have that if $(X_i)_{1\leq i\leq s}$ is a centered Gaussian vector and $\sigma^2=\max_i \Var(X_i)$, then 
\begin{equation}\label{emax}
\E(\max_{1\leq i\leq s} X_i)\leq \sigma\sqrt{2\log s}.
\end{equation}
This can be found, for example, in the appendix of the book by Chatterjee \cite{Chatterjee}; it follows from a simple union bound.  Nevertheless, repeated carefully chosen applications of \eqref{emax} will suffice to prove our upper bound on $M$.  (Importantly, observe by comparison with \eqref{maxnorm} that independent normals are asymptotically the worst case for the expected maximum.)

We also have
\beq{concnorm}{
\Pr(|\max_{1\leq i\leq s} X_i-\E(\max_{1\leq i\leq s} X_i)|>t)\leq 2e^{-t^2/2\s^2}.
}
See for example Ledoux \cite{Led}.

\subsection{First Visit Time Lemma}
In this section we give a lemma that the first author has used (along with Colin Cooper) many times in the study of the cover time of various models of random graphs. Let $G$ denote a fixed connected graph, and let $u$ be some arbitrary vertex from which a walk $\cW_{u}$  is started. Let $\cW_{u}(t)$ be the vertex reached at step $t$ and let $P_u^{(t)}(x)=\Pr(\cW_{u}(t)=x)$. In the following lemma, $\om=\om(n)$ is an arbitrary function that tends to $\infty$ with $n$ and $T=T_{mix}$ is a mixing time in the sense that for $t\geq T_{mix}$
\begin{equation}\label{mix0}
\max_{u,x\in V}\card{\frac{P_{u}^{(t)}(x)-\pi_x}{\pi_x}} \leq \frac{1}{\om}
\end{equation}
Next, considering the  walk $\cW_v$, starting at $v$, let $r_t=\Pr(\cW_v(t)=v)$ be the probability  that this  walk returns to $v$ at step $t = 0,1,...$. 

For $t\geq 0$, let $\cA_t(v)$ be the event that $\cW_u$ does not visit $v$ in steps $T,T+1,\ldots,t$. The vertex $u$ will have to be implicit in this definition. Let $\p_v$ be the steady state probability of vertex $v$ and 
\beq{Rv}{
R_v=\sum_{t=0}^Tr_t. 
}
\begin{lemma}\label{FVL}
Suppose that 
\beq{Tpi}{
T\pi_v=o(1).
}
Then for all $t\geq T$,
\begin{equation}
\label{frat}
\Pr(\cA_t(v))=\exp\set{-\frac{\p_v(1+O(T\p_v))t}{R_v}}+o(Te^{-c\l t/T}),
\end{equation}
for some absolute constant $c>0$.
\end{lemma}
In the lemma as used by Cooper and Frieze \cite{CF1} -- \cite{CFL}, there was a technical condition that has been removed by Manzo, Quattropani and Scoppola \cite{MQS} and we have taken advantage of this improvement to the lemma.
\subsection{Effective resistance on $K_2$}
Recall that the emerging giant can be modeled as a collection of independent Poisson Galton Watson trees attached to $K_2$.  Our proof will depend on a bound on the effective resistance of $K_2$ and then show that this bound suffices to analyze the effective resistance within the Galton Watson trees.  Recall that we think of the graph as an electrical network where each edge is a resistor of resistance one. 

There are several steps to the analysis and we give an outline here. The main result of the section is Lemma \ref{maxReff}.
\begin{enumerate}[(a)]
\item At the top level we bound effective resistance between $v,v_0\in V(K_2)$ using the commute time identity, \eqref{Reffdef}.
\item We observe that a random walk on $K_2$ is rapidly mixing and so bounding commute times reduces to bounding the expected time to visit $v_0$ from the steady state using the First Visit Lemma.  We transform $K_2$ into a related graph $\wK_{2}$ to ensure that \eqref{Tpi} holds, and such that a bound on resistance for $\wK_2$ yields a bound on resistance in $K_2$.
\item To apply Lemma \ref{FVL} we need to bound $R_v$, the expected number of returns to a vertex $v$ within the mixing time $T$. Almost all vertices of $K_2,\wK_2$ are far from short cycles and so their local neighborhoods are trees. We prune these trees so that they induce binary trees in $K_1$. This just simplifies some calculations. Pruning increases $R_v$ and effective resistances and thus it suffices to bound $R_v$ on these pruned trees.
\item Having control of the $R_v$ allows Lemma \ref{FVL} to control commute times. When we apply this lemma in Section \ref{UpperBound}, we find some minor correlation problem. This will be handled with the use of the edge-deletion graphs $\wK_{2,e}$ defined below.
\end{enumerate}

\paragraph{Transforming $K_2$}  Let $\ell_1=\rdup{\k\log N/\log\log N}$. We replace each such path of length $\ell$ in $K_2$ by one of length $\rdup{\ell/\ell_1}\ell_1$. Rayleigh's Law (\cite{DS}, \cite{LPW}) states that increasing the resistance of any edge increases all effective resistances. Placing a vertex in the middle of an edge has the same effect as that of increasing the resistance of that edge. This implies that all resistances between vertices are increased by this change of path length. Now every path has a length which is a multiple of $\ell_1$ and so if we replace paths, currently of length $k\ell_1$ by paths of length $k$, then we change all resistances by the same factor $\ell_1$.  We let $\wK_{2}=(\wV,\wE)$ denote the graph obtained in the above manner and let $\wR_{\eff}$ denote effective resistance in $\wK_{2}$. 

For $e\in E(K_1)$ we let $R_{\eff,e}$ denote effective resistance in $K_2-E(P_e)$. In addition, for each $e\in E(K_1)$ we shorten paths $P_f,f\neq e$ in $K_2-E(P_e)$. The graph obtained is $\wh K_{2,e}=(\wV,\wE_e)$. Let $\wR_{\eff,e}\geq \wR_{\eff}$ denote effective resistance in $\wh K_{2,e}$.
\begin{Remark}\label{resrem}
From our construction, we see that $\wR_{\eff,e}$ is independent of the length of $P_e$. The usefulness of this construct will become apparant when we estimate the size of the sets $U^{i,j,k}$ in Section \ref{321}.
\end{Remark}
Suppose next that we arbitrarily orient the induced paths $P_e,e\in E(K_1)$ from $h_e$ to $t_e$ where $e=\set{h_e,t_e}$. For $v\in V(K_2)\setminus V(K_1)$, we let $e_1(v)$ denote the edge of $K_1$ whose division includes $v$.  %We then let $R_{\eff,e}(v,v_0)=R_{\eff,e}(t_e,v_0)$ and let $\wR_{\eff,e}(v,v_0)$ be what we get by the path shortening as explained above.
We note that
\beq{RR}{
  R_{\eff}(v,v_0)\leq R_{\eff,e_1(v)}(t_{e_1(v)},v_0)\leq \ell_1\wR_{\eff,e_1(v)}(t_{e_1(v)},v_0)\text{ for all }v\in V(K_2)\setminus V(K_1).
}
For $k\geq 1$ we let 
\begin{align*}
%  A_k&=\set{v\in V(K_1): R_{\eff}(v,v_0)\geq \k k}\\
\wh A_k&=\set{e\in E(K_1): \:\wR_{\eff,e}(t_e,v_0)\geq \frac{\k k}{\ell_1}}.
\end{align*}
%We note that
%\[
%A_k\subseteq \wh A_{\rdown{k/\ell_1}}\text{ for }k\geq 1.
%\]
Most vertices in $K_1$ have tree-like neighborhoods. We will define the notion of a tree-like vertex formally below. Suffice it to say at the moment that w.h.p. there are at most $\log^{100}N$ vertices that are not tree-like.
\begin{lemma}\label{maxReff} 
If $t_e\in V(K_1)$ is tree-like, then
\beq{kdef}{
\Pr(e\in \wh A_k|K_1)\le e^{-(2-o(1))\k k},\quad\frac{\ell_1}\k\leq k\leq 2\log N.
}
Here we are conditioning on the output of Step 1 in \textsc{giantconstruction}; the probability space is just over the randomness in Step 2.
\end{lemma}
\begin{proof}
We use the commute time identity \eqref{Reffdef} (\cite{DS}, \cite{LPW}) for a random walk $\chW_{e}$ on the graph $\wK_{2,e}$, to write for $v\in e\in E(K_1)$,
\beq{stime}{
2\wR_{\eff,e}(v,v_0)|\wE_e|=\t(v,v_0)+\t(v_0,v),
}
where $\t(v,w)$ is the expected time for $\chW_e$ to reach $w$ when started at $v$.

The proof of this lemma is unfortunately quite long. We break it up into a sequence of claims, that we will verify subsequently. In what follows $v\in V(K_1)$ will be fixed and $e$ will be a fixed edge of $K_1$ that contains $v$.
\begin{claim}\label{cl1a}
W.h.p., the mixing time $\wT_{mix}$ of $\chW_{e}$ is $O((\log\log N)^2\log N)$, assuming we take $\om=N$ in \eqref{mix0}.
\end{claim}
For vertices $v,w\in V(K_1)$ we bound $\t(v,w)$ by $\wT_{mix}$ plus the expected time to reach $w$ from the steady state of $\chW_e$.
\begin{claim}\label{cl2}
The expected time for $\chW_e$ to reach vertex $v$ from the steady state is $O(R_v/\p_v)$, where $R_v$ is as defined in \eqref{Rv}.
\end{claim}
Fix $e\in E(K_1)$. For a vertex $v\in V(K_1)$ we let let ${N}_v$ (the {\em neighborhood}) be the subgraph of $K_1$ induced by the set of vertices on paths of length at most $L=1000\log\log N$ in $K_1-e$. Then let $\wh{N}_v$ be the subgraph of $\wK_2-e$ that is obtained from ${N}_v$  through the execution of Step 2 and the subsequent shortening of paths that creates $\wK_2$. 

We say that $v\in V(K_1)$ is {\em tree-like} if $N_v$ (and hence $\wh{N}_v$) induces a tree. 
\begin{claim}\label{cl2a}
W.h.p.,
\begin{enumerate}[(a)]
\item For all $v$, $N_v$ contains at most one cycle.
\item The number of non-tree-like vertices is at most $\log^{100}N$.
\end{enumerate}
\end{claim}
In view of this claim, we will mainly focus on tree-like vertices and deal with the non-tree-like vertices fairly crudely. Let $T_v$ denote the tree induced by $N_v$ and let $\wh{T}_v$ denote the tree induced by $\wh{N}_v$. Let $\wh B_v=B_v$ (the {\em boundary}) denote the leaves of $\wh T_v$ (equivalently, the leaves of $T_v$).
\begin{claim}\label{cl3}
If $w\in \wh{B}_v$ then the expected number of visits to $v$ from $w$ in $\wh K_2$, in time $\wT_{mix}$, is $o(1)$.
\end{claim}
Thus, if we make $\wh{B}_v$ into absorbing states for the walk $\chW_e$ then $R_v$ is the expected number of returns before absorption, plus $o(1)$. So let $\tR_v$ be the expected number of visits to $v$ before the walk is absorbed into $\wh{B}_v$. Thus $R_v\leq \tR_v+o(1)$. Next let $p_{esc}(v)$ be the {\em escape probability} i.e. the probability that a random walk started at $v$ doesn't return to $v$, before being absorbed. Then 
\beq{Doyle}{
\tR_v=\frac{1}{p_{esc}(v)}\text{ and }p_{esc}(v)=\frac{1}{D_v\tR_{\eff}(v,\wh{B}_v)}.
}
Recall that $D_v$ denotes the degree of vertex $v$. For a proof of the second 
equation in \eqref{Doyle}, see Doyle and Snell \cite{DS}, Section 1.3.4.

We now prune $T_v$: moving level by level from the neighbors of the root $v$, we prune $T_v$ so that we obtain a tree of depth $L$ in which every vertex other than the root or the leaves has degree three. It is possible that the root $v$ already has degree two. Remember that we have deleted one edge $e$, incident to $v$. We denote the pruned tree by $\tilde T_v$. Rayleigh's principle and equation \eqref{Doyle} show that the pruning decreases the escape probability and increases the expected number of returns which is now denoted $\tilde R_v$. (Note that the pruning can only reduce the expected number of visits in Claim \ref{cl3}.)  Let $T^*_v$ be the subtree of $\wh T_v$ corresponding to $\tilde T_v$.

An edge $f\in E(K_1)$ gives rise to a path $P_f$ in $\wK_2$ and let $\psi(f)=\wh \ell(P_1)-1$, where $\wh \ell(\cdot)$ denotes $\lceil\ell(\cdot)/\ell_1 \rceil$.  Note: our definition of $\ell_1$ means that w.h.p. almost all of the paths $P_f$ in $\wh K_2$ consist of a single edge and for these $\psi=0$. Also let $\psi(v)=\sum_{f\in E(T^*_v)}\psi(f)$. Let $W_s=\set{v\in V(K_1):\psi(v)\leq s\k}$.
\begin{claim}\label{cl4}
W.h.p., if $v\in V(K_1)$ then
\begin{enumerate}[(a)]
\item $\Pr(v\notin W_s)\leq  \exp\set{-\frac{s\k(\k\e(1-\e)\log N-1000(\log\log N)^2)}{\log\log N}}$.
\item If $v\in W_s$ and $e\in E(K_1)$ and $t_e=v$ then 
\[
\wR_{\eff,e}(t_e,\wh B_{t_e})\leq \begin{cases}\frac12&s=0.\\\frac{s\k}{4}+\frac12&s\geq 1.\end{cases}
\]
\end{enumerate}
\end{claim}
In summary, if $\wR_{\eff,e}(t_e,v_0)>k\k/\ell_1$ then $t_e\notin W_s$ where $s\k/4+1/2=k\k/\ell_1\geq 1$. Therefore, for $k$ as in \eqref{kdef},
\mults{
\Pr(e\in \wh A_k\mid K_1)\leq \Pr(v\notin W_s)\leq \exp\set{-\frac{(4k\k-2\ell_1)((1-\e)\log N-1000(\log\log N)^2)}{\ell_1\log\log N}}\\
 =e^{-(2-o(1))\k k}.
}
This would complete the proof of Lemma \ref{maxReff}.  We must now substantiate our claims.
\paragraph{Proof of Claim \ref{cl1a}}
 For a graph $G=(V,E)$, let $e_G(S)$ denote the number of edges contained in the set $S\subseteq V$ and $e_G(S:\bar S)$ be the number of edges with exactly one end in $S$. For a graph $G$ and $S\subseteq V$ let $\F_G(S)=\frac{e_G(S:\bar{S})}{D(S)}$ where $D(S)$ is the sum of degrees of vertices in $S$. The {\em conductance} $\F_G$ of $G$ is equal to $\min_{D(S)\leq |E|}\F_G(S)$. It is shown in \cite{DKLP}, Lemma 3.5 that w.h.p. $\F_{K_1}\geq c_1$, for some absolute constant $c_1>0$. We need the conductance of $K_1-f$ where $f$ is an arbitrary edge of $K_1$. 
\begin{claim}\label{density}
In $K_1$, w.h.p., $e(S)\leq |S|$ for $|S|\leq \log^{1/2}N$.
\end{claim}
Assume this claim for now, and condition on the event in the claim. Let $\tilde e(S:\bar S)$ denote the edges other than $f$ between $S$ and $\bar S$. Then we have $\tilde e(S:\bar S)\geq e(S:\bar S)-1$. If $2\leq |S|\leq \log^{1/2}N$ then because the minimum degree in $K_1$ is at least 3, $\tilde e(S:\bar S)\geq e(S:\bar S)-1\geq |S|-1$. If $|S|\geq \log^{1/2}N$ then $e(S:\bar S)\geq 3c_1|S|$ and then $\tilde e(S:\bar S)\geq \brac{3c_1-\frac{1}{\log^{1/2}N}}|S|$ and so the conductance of $K_1-f$ is at least $c_1/2$.

The conductance of $\wK_{2,e}$ is at least $\frac{c_1}{2}\cdot \frac{1}{2\log\log N}$ because each edge of $K_1-e$ is replaced by a path of length at most $2\log\log N$. Finally note that for a random walk on a graph $G$, we have that after $t$ steps $\max\set{|P_u^{(t)}(x)-\p_x|}\leq \brac{1-\frac{\F_G^2}2}^t$, see for example \cite{LPW}. Putting $t=C(\log\log N)^2\log N$ yields the claim, for $C$ sufficiently large.
\\{\bf End of Proof of Claim \ref{cl1a}}
\paragraph{Proof of Claim \ref{cl2}}
This will follow from Lemma \ref{FVL} applied to the random walk on $\wK_{2}$, once we have verified \eqref{Tpi}. Here $T=O(\log^{1+o(1)}N)$ and $\max \p_v=O\bfrac{\log N}{N}$ and so $T\p_v=O\bfrac{\log^{2+o(1)}N}{N}$. Then we have, from \eqref{frat}, that the expected time to reach $v$ is of order
\[
\sum_{t\geq T}\Pr(\cA_t(v))= \sum_{t\geq T}\brac{\exp\set{-\frac{\p_v(1+O(T\p_v))t}{R_v}}+o(Te^{-c t/T})} \leq \frac{(1+o(1))R_v}{\p_v}.
\]
{\bf End of Proof of Claim \ref{cl2}}
\paragraph{Proof of Claim \ref{cl2a}}
For this claim we use the configuration model of Bollob\'as \cite{BM} as applied to $K_1$. We note that w.h.p. $\La\approx 2\e$ in Step 1, see \eqref{valmu}. And also that $N_{k\geq 3}\approx N$.

(a) If $N_v$ contains more than one cycle, then $K_1$ contains a set $S$ of at most $s\leq 4L$ vertices that contain at least $s+1$ edges. The probability $\Pi$ of this can be bounded as follows: let $\f= \frac{\La^3e^{-\La}}{6}\approx \frac{(2\e)^3}{6}$ be the probability that $Poisson(\La)\geq 3$. 

In the following, $s$ is the size of $S$. Then $3s\leq D\leq M_1$ is the total degree of $S$ and $d_1,\ldots,d_s$ are the individual degrees. Here $M_1\approx 2N$ will be a high probability bound on $|E(K_1)|$. We multiply by the probablity $\prod_{i=1}^s\frac{\La^{d_i}e^{-\La}}{d_i!\f}$ that these are the degrees. Then we choose $2s+2$ configuration points and pair them up in $\binom{D}{2s+2} \frac{(2s+2)!}{(s+1)!2^{s+1}}$ ways. The final term $\bfrac{s+1}{3N}^{s+1}$ bounds the probability of the pairings. Thus
\[
\Pi\leq \sum_{s=4}^{4L}\binom{N}{s}\sum_{D=3s}^{M_1} \sum_{\substack{d_1+\cdots+d_s=D\\d_1,\ldots,d_s\geq3}} \brac{\prod_{i=1}^s\frac{\La^{d_i}e^{-\La}}{d_i!\f}}\binom{D}{2s+2} \frac{(2s+2)!}{(s+1)!2^{s+1}}\bfrac{s+1}{3N}^{s+1}.
\]
But {(i)  $\prod_{i=1}^s\frac{\La^{d_i}e^{-\La}}{\f} =\frac{\La^De^{-\La s}}{\f^s}$, (ii) $\binom{D}{2s+2}\leq \frac{D^{2s+2}}{(2s+2)!}$ and \\ (iii) $\sum_{\substack{d_1+\cdots+d_s=D\\d_1,\ldots,d_s\geq3}}\prod_{i=1}^s\frac{1}{d_i!}\leq \frac{1}{6^s(D-3s)!} \sum_{\substack{d_1+\cdots+d_s=D\\d_1,\ldots,d_s\geq3}}\binom{D-3s}{d_1-3,\ldots,d_s-3}$}.
So,
\[
\Pi\leq \sum_{s=4}^{4L} \bfrac{Ne}{s}^s  \frac{e^{-\La s}}{\f^s(s+1)!2^{s+1}} \bfrac{s+1}{3N}^{s+1} \sum_{D=3s}^{M_1} \frac{\La^DD^{2s+2}}{6^s(D-3s)!} \sum_{\substack{d_1+\cdots+d_s=D\\d_1,\ldots,d_s\geq3}}\binom{D-3s}{d_1-3,\ldots,d_s-3}
\]
But (i) $1/\f\approx 6/(2\e)^3$ and $\La\approx 2\e$ and $\sum_{\substack{d_1+\cdots+d_s=D\\d_1,\ldots,d_s\geq3}}\binom{D-3s}{d_1-3,\ldots,d_s-3}=s^{D-3s}$. So,
\[
\Pi\leq \sum_{s=4}^{4L} \bfrac{Ne}{s}^s \frac{e^{o(s)}6^s}{(2\e)^{3s}(s+1)!2^{s+1}}\bfrac{s+1}{3N}^{s+1} \sum_{D=3s}^{M_1} \frac{((2+o(1))\e)^DD^{2s+2}s^{D-3s}}{6^s(D-3s)!}.
\]
Next let $u_D=\frac{((2+o(1))\e)^DD^{2s+2}s^{D-3s}}{(D-3s)!}$. Then, 
\begin{align*}
\frac{u_{D+1}}{u_D}\leq \frac{(2+o(1))\e s}{D-3s}\bfrac{D+1}{D}^{2s+2}\leq \frac{(2+o(1))e^{(2s+2)/D}\e s}{D-3s} \leq \frac12 \text{ if }D\geq 3s+10\e s.
\end{align*}
and so if $D_0=3s+10\e s$ we see that $u_{D_0}\geq \sum_{D>D_0}u_d$ and then 
\begin{align*}
\Pi&\leq 2\sum_{s=4}^{4L} \bfrac{Ne}{s}^s \frac{e^{o(s)}}{(2\e)^{3s}(s+1)!2^{s+1}}\bfrac{s+1}{3N}^{s+1} \sum_{D=3s}^{D_0}\frac{(2\e)^DD^{2s+2}s^{D-3s}}{(D-3s)!}\\
&\leq \frac{e^{O(1)}s}{N}\sum_{s=4}^{4L}\bfrac{e^2 (3s+10\e s)^{6+20\e+2/s}s^{10\e}}{s}^s=o(1).
\end{align*}
(b) The number of non-tree-like vertices is at most the number of vertices that are within $L$ of a cycle of length at most $L$. We can bound the expected number of such vertices as follows: we choose $s$ vertices for the cycle and then another $t$ for the path in $\binom{N}{s}\binom{N}{t}s!t!$ ways. We sum over the degree sequence of the chosen vertices. The factor $\frac{d_{i-1}d_i}{2N}$ bounds the probability the path plus cycle exists.
\begin{align*}
&\sum_{s,t=4}^{L}\binom{N}{s}\binom{N}{t}s!t!\sum_{d_i\geq 3,i\in[s+t]}\prod_{i=1}^{s+t}\brac{\frac{\La^{d_i}e^{-\La}}{d_i!\f}\times\frac{d_{i-1}d_i}{2N}}\text{ where }d_{0}=d_s\\
\leq &\sum_{s,t=4}^{L}\sum_{d_i\geq 3,i\in[s+t]} \prod_{i=1}^{s+t}\frac{\La^{d_i}e^{-\La}}{(d_i-2)!\f}\\
\leq &\sum_{s,t=4}^{L}\brac{\sum_{d=3}^\infty \frac{\La^{d}e^{-\La}}{(d-2)!\f}}^{s+t}\\
\leq &\sum_{s,t=4}^{L}\brac{6\sum_{d-3=0}^\infty \frac{\La^{d-3}}{(d-3)!}}^{s+t}\\
&\leq \log^{5000}N.
\end{align*}
The claim follows from applying the Markov inequality.\\
{\bf End of Proof of Claim \ref{cl2a}}\\
\paragraph{Proof of Claim \ref{cl3}}
We bound the number of returns as follows. Consider a random walk $\cX$ on $\set{0,1,2,\ldots}$ where we start the walk at $0$ and when at $0< i<L$ we go to $i+1$ with probabilty 1/3 and to $i-1$ with probability 2/3. Whenever we are at 0 we move to 1 on the next move. Here $0$ represents an arbitrary boundary vertex and $L$ represents $v$. At each point of the walk on $\wT_v$ where we are at a vertex of $K_1$, we have probability at most 1/3 of moving closer to $v$.

Now consider a time $t$ when $\cX(t)=L/2$. If $\cX(t+L/4)\geq L/2$ then at least $L/8$ of these $L/4$ moves must be in the increasing direction. But the Chernoff bounds then imply that
\[
\Pr\brac{\cX\brac{t+\frac{L}4}\geq \frac{L}2}\leq \Pr\brac{Bin\brac{\frac{L}4,\frac13}\geq \frac{L}8}\leq \exp\set{-\frac{L}{12}\times \frac{1}{27}}\leq \frac{1}{\log^3N}.
\]
It follows from this that the probability a walk from the boundary reaches $v$ in $T$ steps is at most $T/\log^3N$ and then the expected number of visits is at most $T^2/\log^3N=o(1)$.
\\ {\bf End of Proof of Claim \ref{cl3}}

\paragraph{Proof of Claim \ref{cl4}}
(a) For an edge $e$ of $\tilde T_v$, we have that $\Pr(\psi(e)\geq t)\leq (1-\e(1-\e))^{t\ell_1}$, a probabilistic bound on the length of the path $P_e$ ins Step 2 of {\sc{giantconstruction}} (see \eqref{valmu}). The $\psi$ values of each such edge are independent and so as $\tilde T_v$ contains $m\leq 3\cdot 2^{1000\log\log N}$ edges then
\begin{align}
\Pr(v\notin W_s)&\leq \sum_{s_1+\cdots+s_m=t\geq s\k}\prod_{i=1}^m(1-\e(1-\e))^{s_i\ell_1}\nn\\ &=\sum_{t\geq s\k}\binom{m+t-1}{t-1}(1-\e(1-\e))^{t\ell_1}\nn\\
&\leq \sum_{t\geq s\k}\brac{\frac{(m+t)e}{t}\cdot \exp\set{-\frac{\k\e(1-\e)\log N}{\log\log N}}}^{t} \label{skt}
\end{align}
Let $u_t$ denote the summand in \eqref{skt}. We have that if $s\k\leq m$ then
\mult{eq1}{
\sum_{t=s\k}^mu_t\leq \sum_{t=s\k}^m \brac{2me\cdot \exp\set{-\frac{\k\e(1-\e)\log N}{\log\log N}}}^{t}\leq\\
\sum_{t=s\k}^m \exp\set{-\frac{t\brac{\k\e(1-\e)\log N-700(\log\log N)^2}}{\log\log N}}\\
\leq\exp\set{-\frac{s\k(\k\e(1-\e)\log N-800(\log\log N)^2)}{\log\log N}}.
}
And,
\mult{eq2}{
\sum_{t\geq \max\set{s\k,m}}u_t\leq \sum_{t\geq \max\set{s\k,m}}\brac{2e\cdot \exp\set{-\frac{\k\e(1-\e)\log N}{\log\log N}}}^{t}\\
\leq\exp\set{-\frac{s\k(\k\e(1-\e)\log N-800(\log\log N)^2)}{\log\log N}}.
}
Part (a) of the claim follows from \eqref{eq1} and \eqref{eq2}.

(b) %As far as resistance is concerned we can replace induced paths by edges of the same length.
Given $T^*_v$ with $\psi(v)=s$ we modify it in such a way that the expected number of returns increase and then bound this as claimed.  Roughly speaking, we concentrate all the resistance at the induced paths incident with $v$; by proving that this only increases effective resistance, it allows us to reduce the problem of bounding the effective resistance to this case.

Suppose then that $v\neq w\in V(K_1)\cap V(T_v^*)$ and $w$'s neighbors in $K_1$ are $w_0,w_1,w_2$ where $w_0$ is the one closer to $v$ than $w$ on the tree $T_v^*$. Suppose also that $\psi(\{w,w_1\})+\psi(\{w,w_2\})>0$. We transform $T_v^*$ by increasing the length of the path from $w$ to $w_0$ by $\psi(\{w,w_1\})+\psi(\{w,w_2\})$ and reducing the lengths of the paths joining $w$ to $w_1$ and $w$ to $w_2$ to be single edges so that $\psi(w,w_1)=\psi(w,w_2)=0$. This preserves the sum of $\psi$ values and we claim that $\wR_{\eff,e}(v,\wh B_v)$ does not decrease. In this way, $\wR_v$ does not decrease, see \eqref{Doyle}. To see this, let $\r(w),w\in V(T^*_v)$ be the effective resistance between $w$ and $\wh B_v$ as measured in the sub-tree with root $w$. Let $w_0,w_1,w_2$ be as before and let $w_3$ be the other neighbor of $w_0$ further from $v$ (if it exists). Before the transformation, we have
\beq{bad1}{
\frac{1}{\r(w_0)}=\frac{1}{\ell(w_0,w)+\frac{1}{\frac{1}{\ell(w_0,w_1)+\r(w_1)}+\frac{1}{\ell(w_0,w_2)+\r(w_2)}}} +\frac{1}{\ell(w_0,w_3)+\r(w_3)}
}
and after the transformation we have
\beq{bad2}{
\frac{1}{\r(w_0)}=\frac{1}{\ell(w_0,w)+\ell(w,w_1)+\ell(w,w_2)-2+\frac{1}{\frac{1}{1+\r(w_1)}+\frac{1}{1+\r(w_2)}}} +\frac{1}{\ell(w_0,w_3)+\r(w_3)}
}
The R.H.S. of \eqref{bad2} is at most the R.H.S. of \eqref{bad1}. This follows from the inequality
\beq{bad3}{
\a+\b+\frac{1}{\frac{1}{\g}+\frac{1}{\d}}-\frac{1}{\frac{1}{\a+\g}+\frac{1}{\b+\d}}\geq 0.
}
After multiplying through by $(\a+\b+\g+\d)(\g+\d)$ we obtain an expression with only positive terms. We apply \eqref{bad3} with $\a=\ell(w,w_1)-1,\b=\ell(w,w_2)-1,\g=\r(w_1),\d=\r(w_2)$.

Proceeding in this way, we end up with a tree in which all maximal induced paths in $T_v^*$ are of length one, except for the one incident with $v$. Furthermore, $\psi$ is unchanged and resistance is not decreased by this transformation. The sum of the lengths of the maximal induced paths incident with $v$ is then $\psi(v)+2$ (recall that $v$ has degree 2 in $T_v^*$)

Finally, we balance the lengths of these two paths incident with $v$ by replacing the path lengths at $v$ by $1+\lceil\frac{\psi(v)}{2}\rceil$ and $1+\rdown{\frac{\psi(v)}{2}}$. This increases resistance because for positive integers $x,y$, we have $\frac1x+\frac1y\geq \frac{1}{\rdup{(x+y)/2}}+\frac{1}{\rdown{(x+y)/2}}$.

Note next that the effective resistance between the root of a binary tree and its leaves is at most one. To see this we let $R_d$ be the effective resistance if the depth is $d$. Then we have
\[
R_d=\frac{1}{\frac{1}{R_{d-1}+1}+\frac{1}{R_{d-1}+1}}=\frac{R_{d-1}+1}{2}.
\]
It then follows that 
\[
\wR_{\eff,e}(v,\wh B_v)\leq \frac{1}{\frac{1}{1+\lceil\frac{\psi(v)}{2}\rceil}+\frac{1}{1+\rdown{\frac{\psi(v)}{2}}}} =\frac{\brac{1+\lceil\frac{\psi(v)}{2}\rceil}\brac{1+\rdown{\frac{\psi(v)}{2}}}}{2+\psi(v)} \frac{\brac{1+\frac{\psi(v)}2}^2}{2+\psi(v)}= \frac{\psi(v)}4+\frac12.
\]
\\ {\bf End of Proof of Claim \ref{cl4}}
\paragraph{Proof of Claim \ref{density}}
Let $\f\approx\frac{\La^3e^{-\La}}{6},\La\approx 2\e$ be the probability that $Poisson(\La)\geq 3$. For a set $S\subseteq V(K_1)$ with $|S|=s$, we have
\begin{align}
\Pr(e(S)\geq s+1)&\leq \sum_{D\geq 3s}\sum_{d_1+\cdots+d_s=D}\prod_{i=1}^{s}\frac{\La^{d_i}}{d_i!\f} \times 2^D\bfrac{D}{N}^{s+1}\nn\\
&\leq \sum_{D\geq 3s}2^D\bfrac{D}{N}^{s+1}\times \frac{\La^D}{\f^{s}}\times \frac{s^D}{D!}.\label{eS}
\end{align}
{\bf Explanation:} Let $M_1=|E(K_1)|$ and let $D=D(S)$ denote the sum of the degrees in $S$ and $\sum_{d_1+\cdots+d_s=D}\prod_{i=1}^s\frac{\La^{d_i}}{d_i!\f}$ bounds the probability that this sum is $D$. To bound the probability that $e(S)\geq s+1$ we have to choose some subset of the $D$ configuration points of size $s+1$ that pair with configuration points in $S$. We bound the probability that such a set of configuration points exist  by $2^D\bfrac{D}{N}^{s+1}$. Note here that $M_1\geq 3N/2$ and the probability that a configuration point of $S$, pairs with another such point is bounded by $(D-1)/(2M-1)$, conditional on previous pairings of points in $S$. Finally, we bound $\sum_{D\geq 3s}\sum_{d_1+\cdots+d_s=D}\prod_{i=1}^s\frac{1}{d_i!}$ by $\frac{s^D}{D!}$.

Letting $u_D=2^D\La^DD^{s}\frac{s^D}{D!}$ we see that 
\[
\frac{u_D}{u_{D+1}}\leq 2\times e^{s/D}\times (2+o(1))\e\times \frac{s}{D}\ll 1.
\]
So, 
\mults{
\Pr(\exists |S|\leq \log^{1/2}N:e(S)\geq |S|+1)\leq \frac{3\log^{1/2}N}{ N}\sum_{s=4}^{\log^{1/2}N}\binom{N}{s}\brac{e^{o(1)}\times 8\times 3\times 6\times \frac{e^3}{27}\times \bfrac{s}{N}}^s\\
\leq \frac{3\log^{1/2}N}{N}\sum_{s=4}^{\log^{1/2}N}150^s=o(1).
}
\\ {\bf End of Proof of Claim \ref{density}}

This completes the proof of Lemma \ref{maxReff}.   (Because there are so few non-tree-like vertices, for such $v$ we will bound $\wR_{\eff,e}(v,v_0)$ by the diameter $O\bfrac{\log N}{\e}$ of $K_2$.)
\end{proof}
\section{Proof of Theorem \ref{Mexp}}\label{prf}
Theorem \ref{themg} allows us to work with $H$ instead of $C_1$, and we assume from now on that $H$ has the following properties that have been shown or claimed to hold w.h.p. above, namely:

Assumed Properties of $H$: APOH
\begin{enumerate}[(i)]
  \item $|V(K_1)|\approx 4N/3$,
  \item $|E(K_1)|\approx 2N$,
  \item $|V(K_2)|\approx 2\e^2 n$,
  \item $|E(K_2)|\approx 2\e^2 n$,
  \item $|V(H)|\approx 2\e n$,
  \item $|E(H)|\approx 2\e n$.
  \item There are between $\frac12c_1N^{1-\g+O(\e)}$ and $2c_2N^{1-\g+O(\e)}$ trees of depth at least $\g\e^{-1}\log N$ and there are no trees of depth exceeding $\frac{2\log N}{\e}$.
\end{enumerate}
In what follows, we may write in terms of unconditional probabilities and expectations, but these will refer to the GFF and will assume that $H$ is a fixed graph with property APOH. There are some places where we have to prove further properties of $H$, but we will be sure to flag them.
\subsection{Lower Bound}
It turns out that for the lower bound, it suffices to consider the maximum over a very restricted set, consisting just of a single vertex from each sufficiently deep tree.
\begin{lemma}
\[
\E\brac{\max_{v\in V(G)}\eta_v}\geq (1+o(1))\frac{\log(\e^3n)}{(2\e)^{1/2}}.
\]
\end{lemma}
\begin{proof}
We first identify a subset of verties on which the GFF \emph{behaves as having} independent components and then produce a lower bound using Slepian's comparison \eqref{maxcomp}, combined with \eqref{maxnorm}. 
Consider the set of Galton-Watson trees attached to $H$ of depth at least $d={i\e^{-1}}$, $i$ to be chosen. Choose one vertex at depth $d$ from each tree to create $S_d$. It follows from  \eqref{BinTrees} with $\g=i/\log N$, that there will be at least $cN^{1-\g+O(\e)}$ such trees for some constant $c>0$. Let $(\wheta_v)_{v \in S_d}$ be a random vector with i.i.d. $\mathcal{N}(0,\g\e^{-1}\log N)$ entries. Then $\wheta_v-\wheta_w$ has variance exactly $2\g\e^{-1}\log N$ whereas $\eta_v-\eta_w$ has variance at least $2\g\e^{-1}\log N$ (the graph-distance between $v$ and $w$ is at least $2d = 2i\e^{-1} =2\g\e^{-1}\log N$) and so it follows from \eqref{maxcomp} that
\beq{competa}{
\E(\max\set{\eta_v:v\in S_d})\geq \E(\max\set{\wheta_v:v\in S_d}).
}
Applying \eqref{maxnorm}  we see that
\[
\E(\max\set{\wheta_v:v\in S_d})\geq (1+o(1))(2\log(|S_d|)^{1/2}\cdot (\g\e^{-1}\log N)^{1/2}
\]
$\hat \eta_v$ has the same distribution as a standard Gaussian multiplied by $(\g\e^{-1}\log N)^{1/2}$). Using $|S_d| \geq cN^{1-\g+O(\e)}$, we obtain
\mult{loweta}{
\E(\max\set{\wheta_v:v\in S_d})\geq (1+o(1))(2\log(cN^{1-\g+O(\e)}))^{1/2}\cdot (\g\e^{-1}\log N)^{1/2}\\
\approx \frac{(2\g(1-\g))^{1/2}\log N}{\e^{1/2}}.
}
Putting $\g=1/2$ in \eqref{loweta} and applying \eqref{competa} yields
\[
\E\brac{\max_{v\in V(G)}\eta_v}\geq \E\brac{\max_{v \in S_d} \eta_v }\geq (1+o(1))\frac{\log N}{(2\e)^{1/2}}.
\]
Recalling that $N = \e^3n$, this finishes the proof of the lemma.
\end{proof}

The important task is to achieve a matching upper bound.
\subsection{Upper Bound}\label{UpperBound}
We begin with an outline of the proof of the upper bound.

We let $\k:=\lceil 1/\e\rceil$, and will write $\ell_0=\rdup{\log_2\k}$. %We let $L_j$ denote the set of vertices at distance $j$ from $K_2$.
We say that $v\in G$ is a \emph{$d$-survivor} if it has at least one $d$-descendant $x_d(v)$; that is, a vertex $x_d(v)$ such that $\dist(K_2,x_d(v))=\dist(K_2,v)+\dist(v,x_d(v))=\dist(K_2,v)+d$. 

Recall that we have oriented the induced paths $P_e$ from $h_e$ to $t_e$. See the paragraph following Remark \ref{resrem}. Then for each such $e$ and $v\in V(P_e)$ we let $d_1(v)$ denote the distance from $v$ to $V(K_1)$ traversing $P_e$ in the chosen direction. %Let $\t(v)$ denote the vertex $t_e$ of $K_1$ that $P_e$ ends at.
Let $e(v)$ denote the edge of $K_2$ corresponding to the path $P_e$ containing $v$.

Each $v\in V(H)\setminus V(K_2)$ lies in a Galton-Watson tree with a root $w=\r_{GW}(v)\in V(K_2)$ lying on a path created in Step 2 from an edge $e$. Let $d_1(v)=d_1(w)$ and let 
\[
U^{i,0,k}=\set{v\in V(K_2):d_1(v)\in [i\k,(i+1)\k-1],e(\r_{GW}(v))\in \wh A_k\setminus \wh A_{k+1}}
\]
and define for each $1\leq j\leq 2\log N$ and $0\leq i,k\leq 2\log N$ a set $U^{i,j,k}$ by choosing, for each $\k$-survivor in $U^{i,j-1,k}$, an arbitrary $\k$-descendant $x_\k(v)$; these chosen $\k$-descendants comprise $U^{i,j,k}$.  Evidently, we have for $U=\bigcup_{i,j,k\geq 0} U^{i,j,k}$ that
\begin{equation}\label{maxdecomp}
\E(\max_{v\in V} \eta_v)\leq \E(\max_{u\in U} \eta_u)+\E(\max_{v\in V}(\eta_v-\eta_{u(v)})),
\end{equation}
for any function $u:V\to U$. We will bound the two terms on the right hand side separately. Let
$$T_\d=\frac{e^\d\log N}{(2\e)^{1/2}},$$
where $\d=\max\set{10\e,\frac{1}{\log^{1/3}N}}$.
\begin{lemma}\label{extralemmaA}
With the notation introduced above, we have
\begin{equation}
\E(\max_{u\in U} \eta_u)\leq (1+o(1))T_\d.\label{firstterm}\\
\end{equation}
\end{lemma}
\begin{lemma}\label{extralemmaB}
There is a function $u:V\to U$ such that
\begin{equation}
\E(\max_{v\in V}(\eta_v-\eta_{u(v)}))=o(T_{\delta}).\label{secondterm}
\end{equation}
\end{lemma}
Observe that the proof of the upper bound in Theorem \ref{Mexp} follows from \eqref{maxdecomp} and Lemmas \ref{extralemmaA} and \ref{extralemmaB}; it remains just to prove these two Lemmas.

\subsubsection{Proof of Lemma \ref{extralemmaA}}\label{321}
We let $Z_{i,j,k}=\max_{v\in U^{i,j,k}}\eta_v$ and 
\mult{bound1}{
  \E(\max_{v\in U} \eta_v)=\E\left(\max_{0\leq i,j,k\leq 2\log N} Z_{i,j,k}\right)
  \leq T_\d+\sum_{0\leq i,j,k\leq 2\log N} \E\brac{\max(Z_{i,j,k}-T_\d,0)}
  \\= T_\d+\sum_{i,j,k=0}^{2\log N}\int_{t\geq T_\d}\Pr(Z_{i,j,k}\geq t)dt.
}
The bounds on $i,j,k$ follow from Lemmas \ref{longpath}, \ref{xlem1}, \ref{maxReff} respectively.
     
Our task now is to bound the sum of integrals in \eqref{bound1}. In words, the idea is that $U$ is partitioned into smaller pieces $U^{i,j,k}$ such that each piece is of a small enough cardinality such that the Gaussian concentration of $Z_{i,j,k}$ around its mean allows us to control the above integrals. 

Let a vertex of $v$ of $K_2$ be tree-like if the endpoint $t_e$ of the path $P_e$ containing it is a tree-like vertex of $K_1$. Similarly, a vertex of a Galton-Watson tree is tree-like if its root is tree-like.   Now write
\[
U^{i,j,k}=U^{i,j,k}_T\;\dot\cup\; U^{i,j,k}_{N}
\]
where $U^{i,j,k}_T$ and $U^{i,j,k}_{N}$ are those vertices whose GW trees are attached at tree-like and non-tree-like vertices of $K_2$, respectively.

{\bf Case 1: $U^{i,j,k}_T$ for $k_0=\log^{1/2}N\leq k\leq 2\log N$: tree-like vertices}\\
Because we are bounding the sum of integrals on the RHS of \eqref{bound1} it will be safe to ignore events of probability $o(\log^{-3}N)$. So from now on, w.h.p. will mean with probability $1-o(\log^{-3}N)$. We will work assuming that $K_1$ is fixed and satisfies the conditions APOH(i) and (ii)  defined at the beginnning of Section \ref{prf}. We can then focus on $0\leq i,j,k\leq 2\log N$. This is because it follows from Lemmas \ref{longpath}, \ref{lemdepth}(b) and \ref{maxReff} that these bounds hold with probability $1-O(N^{-1-o(1)})$.
\begin{claim}
We have that w.h.p. 
\beq{SizeUj}{
|U^{i,j,k}_T|\leq  O\brac{Ne^{-\e(1-\e)\k(i+j+k)}}\quad \text{ for }0\leq i,j\leq 2\log N,k_0\leq k\leq 2\log N.
}
\end{claim}
\begin{proof}We write
  \[
  |U^{i,j,k}| = \sum\limits_{v\in U^{i,j-1,k}} \textbf{1}_{B_v},
  \]
  where the event $B_v$ is the that vertex $v$ is a $\kappa$-survivor. We have
\mult{SizeUijk}{
  \E(|U^{i,j,k}_T|)=O\brac{\k N(1-\e(1-\e))^{\k i}\cdot e^{-(2-o(1))\k k\th_k}\cdot
    (1-\e(1-\e))^{\k(j-1)}\cdot
    \e e^{-\e\k}}\\ =O\brac{Ne^{-\e(1-\e)\k(i+j+(2-o(1))k)}}.
}
where $\th_k=1_{k\geq \ell_1/\k}$.

{\bf Explanation:} 
For a fixed vertex in $K_2$, the expected number of vertices at level $t$ of a G-W tree rooted at this vertex will be at most $(1-\e(1-\e))^{t}$.  Each vertex $v$ in such a level has probability $\Pr\brac{B_v} \leq \Pr\brac{L_{\kappa} \neq \emptyset}$ of being a $\k$-survivor and we use Lemma \ref{lemdepth} to upper bound $\Pr\brac{L_{\kappa} \neq \emptyset}$ by $O(\e e^{-\e \kappa})$.  Wald's idenity implies that the expected number of vertices in the G-W tree rooted at a fixed vertex lying in $U^{i,j,k}$ is thus $(1-\e(1-\e))^{\k(j-1)}\cdot  \e e^{-\e\k}$.

  In expectation there are $O(\k N(1-\e(1-\e))^{\k i}\cdot e^{-(2-o(1))\k k\th_k})$ vertices $w\in K_2$ for which $e(w)\in \wh A_k$ and $d_1(w)\geq \k i$; here we have used  Lemma \ref{maxReff} to bound the probability that a vertex $w$ for which $d_1(w)\geq \k i$ has $e(w)\in \wh A_k$, and applied Wald's identity as before.  Applying Wald's idenitity a final time gives \eqref{SizeUijk}

  %. For each such vertex, the expected number of vertices at level $t$ of a G-W tree will be at most $(1-\e(1-\e))^{t}$. We apply this with $t = \k(j-1)$. Then we have $\Pr\brac{B_v} \leq \Pr\brac{L_{\kappa} \neq \emptyset}$ and we use Lemma \ref{lemdepth} to upper bound $\Pr\brac{L_{\kappa} \neq \emptyset}$ by $O(\e e^{-\e \kappa})$. This explains the second factor.

%We have used Remark \ref{resrem} to justify independence and  Wald's identity to justify multiplying the factors. The term $\th_k$ amounts to a correction for small $k$.

Equation \eqref{SizeUj} follows from the Markov inequality.  (There are $O(\log^3N)$ choices for $i,j,k$ and there is a factor $e^{(1-o(1))k}\geq e^{(1-o(1))k_0}$ difference between the expressions in \eqref{SizeUj}, \eqref{SizeUijk}.)
\end{proof}

Given \eqref{SizeUj}, we proceed to bound the sum in \eqref{bound1} term by term.  (We wish to show that the sum is $o(T_\d)$.) To bound the probabilities $\Pr\brac{Z_{i,j,k} \geq t}$, we will use the concentration of the maximum of a Gaussian process around its expectation, whereas the expectations $\E(Z_{i,j,k})$ will be simply treated with the union bound.

First we estimate the expectations.
\begin{claim}
For $i,j\geq 0,k \geq k_0$, 
\begin{equation}\label{expect}
\E(Z_{i,j,k}) \leq e^{-\d/2}T_\d.
\end{equation}
\end{claim}
\begin{proof}
For $v\in U^{i,j,k}$, we know that $\eta_v$ has variance at most $\k(i+j+k+1)$ (by the definition of $U^{i,j,k}$, the graph-distance from $v$ to $K_2$ is $\k j$ and $\k(i+k+1)$ comes from the definition of $\wh A_k$. It then follows from \eqref{emax} in Section \ref{normal} and $|U^{i,j,k}| \leq CNe^{-\e(1-\e)\k(i+j+k)}$ that
\begin{equation}\label{EZj}
\tE(Z_{i,j,k})\leq (2\log (CNe^{-\e(1-\e)\k(i+j+k)}))^{1/2}\cdot (\k(i+ j+k+1))^{1/2}.\end{equation}
It follows from $2(xy)^{1/2}\leq  x+y$ that we can write 
\begin{align*}
\tE(Z_{i,j,k})  &\leq(2\e^{-1})^{1/2}%\brac{1+O\bfrac{1}{\log N}}
  (\k\e (i+j+k))^{1/2} (\log (CN)-\e(1-\e)\k (i+j+k)))^{1/2}\\
&\leq
  %\brac{1+O\bfrac{1}{\log N}}
  \frac{(1+7\e)\log (CN)}{(2\e)^{1/2}}\leq e^{-2\d/3}T_\d
\end{align*}
and then $\E(Z_{i,j,k})\leq \e^{-2\d/3}T_\d\leq e^{-\d/2}T_\d$.

{\bf Case 2: $k_0\leq k_0=\log^{1/2}N$: tree-like vertices}\\
We first let $U^{i}$ be the set of vertices $v$ of $K_2$ for which $\dist(v,t_{e(v)})\in [i\k,(i+1)\k-1]$. Given $K_1$ and $|E(K_1)|\approx 2N$ the size of $U^i$ is a binomial random variable with success probability at most $\m^{i\k}\leq (1-\e(1-\e))^{i\k}$. So, w.h.p. 
\[
|U^{i}|\leq 2 Ne^{-\e(1-\e)i\k}+\log^{10}N, \quad\text{for all }0\leq i\leq 2\log N.
\]
The first term come from the Chernoff bounds and the $\log^{10}N$ term is there for the case where the expectation $Ne^{-\e(1-\e)i\k}$ is less than $\log^2N$. In which case we just use the Markov inequality. This estimate is valid conditonal on $\cU$.

For each $v\in U^{i}$ recall that $p=(1-\e)^{1-\e}$ and let $p_j=p^{\k(j-1)}\cdot \e e^{-\e \k}$ bound the probability that $v$ has a descendant at level $j\k$ that is a $\k$-survivor. Then if $U^{i,j}$ denotes the set of descendants of such vertices $v\in U^i$, we have
\[
\tE(|U^{i,j}|)\leq |U^{i}|p_j\leq  (2Ne^{-\e(1-\e)i\k}+\log^{10}N)p_j.
\]
Applying the Chernoff bounds we see that conditional on $\cU$, w.h.p.
\begin{align*}
|U^{i,j}|&\leq 2(2 Ne^{-\e(1-\e)i\k}+\log^{10}N)p_j+\log^{10}N\\
&\leq 4 Ne^{-\e(1-\e)(i+j-1)\k}\cdot\e e^{-\e\k}+2\log^{10}N.
\end{align*}
It then follows using \eqref{emax} that for all $k\leq k_0=\log^{1/2} N$ that
\mult{EZ}{
\tE(Z_{i,j,k})\leq\\
(2\e^{-1})^{1/2}\brac{1+\frac{2\log\log N}{\log N}}(\k\e (i+j+\log^{1/2}N))^{1/2} (\log (4N)-\e(1-\e)\k (i+j))^{1/2}.
}
If now $i+j\leq \frac{1}{100}\log N$ then we see that
\[
\E(Z_{i,j,k})\leq \frac{\k^{1/2}\log N}{9}\leq \frac{T_\d}4.
\]
If $i+j\geq \frac{1}{100}\log N$ then we use $2(xy)^{1/2}\leq  x+y$ and $(i+j+\log^{1/2}N)\leq (i+j)\brac{1+\frac{100}{\log^{1/2}N}}$. Applying this in \eqref{EZ} gives
\[
\tE(Z_{i,j,k})\leq\frac{\brac{1+\frac{101}{\log^{1/2}N}}}{(2\e)^{1/2}}(\log(4N)+4\e\log N)\leq \frac{e^{\d/2}\log N}{(2\e)^{1/2}}\leq \e^{-\d/2}T_\d.
\]
{\bf Case 3: Non-tree-like vertices}\\
Claim \ref{cl2a} says that w.h.p. there are at most  $\log^{100}N$ non-tree-like vertices of $K_1$, we have
\[
\tE(|U_N^{i,j,k}|\mid \text{Claim \ref{cl2a}})=O(\log^{100}Ne^{-\e(1-\e)\k(i+j)})
\]
and so w.h.p. 
\[
|U_N^{i,j,k}|=O(\log^{200}Ne^{-\e(1-\e)\k(i+j)}).
\]
And then, using the bound of $\frac{3\log N}{\e}$ on the diameter from \cite{DKLP2} to bound effective resistance in $K_2$, we have
\mults{
\E(Z_{i,j,k})=O(\log (C\log^{200}Ne^{-\e(1-\e)\k(i+j)})^{1/2}(\e^{-1}\log N)^{1/2})\\
=O((\e^{-1}\log N \log\log N)^{1/2})=o(T_\d)
}
and we can continue as in \eqref{VZj}.

This completes our estimates for $\E(Z_{i,j,k})$. 
\end{proof}

We proceed to estimate the probability the probability that $Z_{i,j,k}$ significantly exceeds its mean.

To estimate this probability we use the Gaussian concentration for the maximum, \eqref{concnorm} in Section \ref{normal}. As already remarked, this inequality will not be affected by the conditioning and it yields
\begin{equation}
\Pr(Z_{i,j,k}\geq \E(Z_{i,j,k})+t) \leq 2\exp\set{-\frac{t^2}{2(i+j+k+1)\k}}\leq 2\exp\set{-\frac{t^2}{13\k\log N}},\label{VZj}
\end{equation}
where in the last inequality we use $i,j,k\leq 2\log N$. Thus,
\mult{int1}{
\int_{t\geq T_\d}\Pr(Z_{i,j,k}\geq t)dt\leq \int_{t\geq T_\d}\exp\set{-\frac{(t-\E(Z_{i,j,k}))^2}{13\k\log N}}dt\\  =\sqrt{13\k\log N}\int_{u\geq \frac{T_\d-{\bf E}(Z_{i,j,k})}{\sqrt{13\k\log N}}}e^{-u^2}du = O\brac{\k^{1/2}\log^{1/2}N\exp\set{-\frac{(T_\d-\E(Z_{i,j,k}))^2}{13\k\log N}}}.
}
Plugging \eqref{expect} into \eqref{int1} we see that 
\begin{align*}
\exp\set{-\frac{(T_\d-\E(Z_{i,j,k}))^2}{13\k\log N}}\leq \exp\set{-\frac{(1-e^{-\d/2})^2T^2_\d}{13\k\log N}} &\leq \exp\set{-\frac{(1-e^{-\d/2})^2e^{2\d}\log N}{26\k\e}} \\
&\leq N^{-c\d^2}
\end{align*}
for some universal constant $c > 0$, as $\k \e \leq 2$, $e^{2\d} \to 1$ and $(1-e^{-\d/2})^2 \approx \delta^2/4$.

So,
\beq{P2}{
\int_{t\geq T_\d}\Pr(Z_{i,j,k}\geq t)dt\leq \k^{1/2}\log^{1/2}N\cdot N^{-c\d^2}\leq N^{-c\d^2}T_\d.
}
Thus
\begin{align}
\sum_{i,j,k=0}^{2\log N} \int_{t\geq T_\d}\Pr(Z_{i,j,k}\geq t)dt \leq 8N^{-c\d^2}T_\d\log^3N  &\leq  \exp\set{-\frac{c\log N}{\log^{2/3}N} +O(1)+ \log\log N}T_\d\notag\\
 &= o(T_\d).\label{first}
\end{align}

\subsubsection{Proof of  Lemma \ref{extralemmaB}}
To prove Lemma \ref{extralemmaB} we let $W_k$ denote the set of vertices whose distance to $K_2$ is divisible by $k$.  Our goal now is to show that a general vertex $v$ is $\eta$-close to some vertex $u(v)\in U$, i.e. as measured by $(\eta_v-\eta_u)$; we will do this by showing that $v$ is $\eta$-close to its  $H$-nearest (as measured by graph distance) ancestor $y\in W_\k$; this will suffice since our choice of $U$ ensures that some vertex $u\in U$ has the property that $y$ is also the $\eta$-closest ancestor of $u$ in $W_\k$.
\begin{figure}[h]
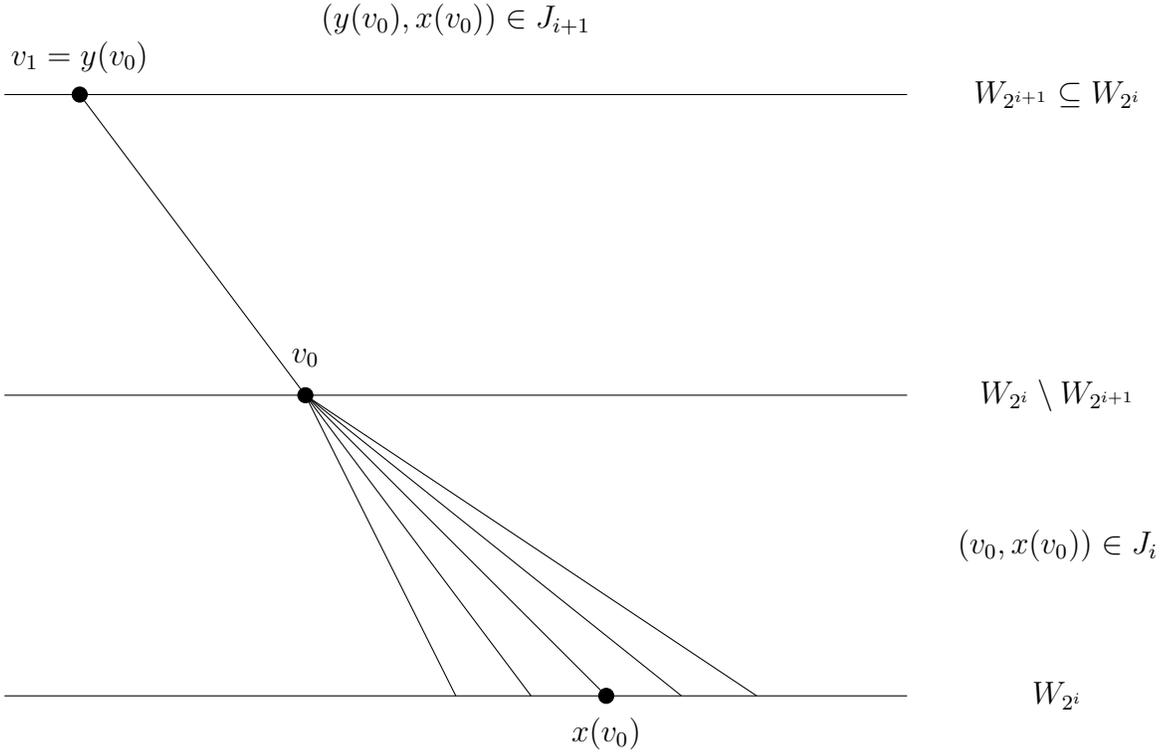

\begin{center}
\pics{1}{
\draw (0,0) -- (12,0);
\draw (0,4) -- (12,4);
\draw(0,8) -- (12,8);
\draw [fill=black] (1,8) circle [radius=0.1];
\node at (1,8.5) {$v_1=y(v_2)$};
\draw [fill=black] (4,4) circle [radius=0.1];
\node at (4,4.5) {$v_2$};
\draw (1,8) -- (4,4);
\draw [fill=black] (8,0) circle [radius=0.1];
\node at (8,-.5) {$x(v_2)$};
\draw (4,4) -- (8,0);
\draw (4,4) -- (9,0);
\draw (4,4) -- (10,0);
\draw (4,4) -- (7,0);
\draw (4,4) -- (6,0);

\node at (6,9) {$(y(v_2),x(v_2))\in J_{i+1}$};
\node at (14,8) {$W_{2^{i+1}}\subseteq W_{2^i}$};
\node at (14,4) {$W_{2^i}\setminus W_{2^{i+1}}$};
\node at (14,0) {$W_{2^i}$};
\node at (14,2) {$(v_2,x(v_2))\in J_i$};
}
\caption{The sets $W_k,J_k$.}
\label{Fig1}
\end{center}
\end{figure}

We will consider sets $J_0,J_1,J_2,\dots,J_{\ell_0},\ell_0=\rdup{\log_2\k}$ of {\em ordered pairs} of vertices in $H$ with the following properties (see Figure \ref{Fig1}): 
\begin{enumerate}[{\bf A}]
\item For $(v_1,v_2)\in J_i$, we have that $v_1,v_2\in W_{2^i}$, and that $v_2$ is a $2^i$-descendant of $v_1$.  
\item $J_0$ is the set of all edges in $H$ that are outside of $K_2$,
\item \label{representative} For each $i$, we have for each $2^{i}$-survivor $v_2\in W_{2^{i}}\setminus W_{2^{i+1}}$ belonging to $\pi_2(J_i)$, that exactly one $2^{i}$-descendant $x(v_2)\in W_{2^{i+1}}$ of $v_2$ is paired in $J_{i+1}$ with its $2^{i+1}$-ancestor $v_1\in W_{2^{i+1}}$.
\item \label{samerep} For all $i$, $\pi_2(J_{i+1})\subset \pi_2(J_{i})$.  (Here $\pi_j$ is the projection function returning the $j$th coordinate of a tuple.)
\end{enumerate}

Notice that pairings $J_0,J_1,\dots,J_{\ell_0}$ with these properties exist by induction; having constructed $J_0,\dots,J_i$, we construct $J_{i+1}$ by choosing pairs via properties \textbf{\ref{representative}} and \textbf{\ref{samerep}}; in particular, for each $2^i$ survivor $v_2$ in $\pi_2(J_i)$ at distance $k2^i$ from $K_2$ for odd $k$, we choose a $2^i$ descendant $x(v_2)$, and add the pair $(v_1,x(v_2))$ to $J_{i+1}$, where $v_1$ is the $2^{i+1}$ ancestor of $x(v_2)$ (and the $2^i$ ancestor of $v_2$).

So we fix some choice of the pairings $J_0,\dots,J_{\ell_0}$. 
We write $\bar J_i$ for the set of unordered pairs which occur (in some order) in $J_i$.  The heart of our argument is the following lemma.
\begin{lemma}\label{lm:paths}
  Given any vertex $v\in V$, let $\alpha(v)$ be its $H$-closest ancestor in $W_\k$. There is a sequence $v=v_0,v_1,v_2,\dots,v_t=\alpha(v)$ such that:
  \begin{enumerate}[(a)]
    \item For each $j=1,\dots,t$, $\{v_{j-1},v_j\}\in \bar J_i$ for some $i$.
    \item For each $i=0,\dots,\ell_0$, at most $1+2(\ell_0-i)$ of the pairs $\set{v_0,v_1},\set{v_1,v_2},\dots,\set{v_{t-1},v_t}$ belong to $\bar J_i$.
  \end{enumerate}
\end{lemma}

\begin{proof}[Proof of Lemma \ref{lm:paths}]\ Fix a vertex $v \in V$. Our goal is to find \emph{a chain} $v=v_0,v_1,v_2,\dots,v_t=\alpha(v)$ such that its  consecutive \emph{links} $\{v_{j-1},v_j\}$ are all in the sets $J_i$ and each set $J_i$ contains at most $1+2(\ell_0-i)$ links. We shall do this recursively and in order to keep track of it, we need the following parameters
  \begin{align*}
    \phi(v)&=\max\set {0\leq i\leq \ell_0\mid v\in W_{2^i}}\\
    \psi(v)&=\max\set {0\leq i\leq \phi(v)\mid v\in \pi_2(J_i)}.
  \end{align*}
 \begin{claim}\label{cl1}
Given any $v$, there is a vertex $a(v)$ such that either
  \begin{enumerate}[(a)]
  \item \label{aorz}  $\phi(a(v))>\phi(v)$ and $(a(v),v)\in J_{\phi(v)}$, or else
  \item $\phi(a(v))=\phi(v)$ and $\psi(a(v))>\psi(v)$, and there exists $z(v)$ such that $(z(v),a(v))$ and $(z(v),v)$ are both in $J_{\psi(v)}$.% for some $i$, 
  \end{enumerate}
\end{claim}
\begin{proof} 
Consider the vertex $v$, and let $i=\phi(v)$.  We consider two cases:\\
  \noindent \textbf{Case 1:} $\psi(v)=\phi(v)$.  In this case, by definition of $\psi(v)$, we have that there is a vertex $a(v)$ such that $(a(v),v)$ in $J_i$.  In particular, as $2^i$ is the largest power of $2$ such that $v\in W_{2^i}$ and $v$ is a $2^i$ descendant of $a(v)$, we have that $a(v)\in W_{2^{i+1}}$; that is, that $\phi(a(v))\geq i+1$, as claimed.\\
  \noindent \textbf{Case 2:} $\psi(v)=j<\phi(v)$.  In this case, by definition of $\psi(v)$, we have that there is a vertex $z$ such that $(z,v)$ in $J_j$.  Now by Property \ref{representative} of the pairings $\{J_i\}$, $z$ has a $2^j$-descendant $a(v)$ which is in $\pi_2(J_{j+1})$; in particular, we have that $\psi(a(v))\geq j+1>\psi(v)$.  (Note for clarity that $a(v)$ and $v$ are at the same distance from $K_1$ in Case 2 and so $\phi(a(v))=\phi(v)$.)  And by Property \ref{samerep}, $a(v)\in \pi_2(J_j)$ as well, and thus $(z,a(v)) \in J_j$, completing the proof of the claim.
This concludes the proof of Claim 1, and thus also Lemma \ref{lm:paths}.
\end{proof}

Observe that Lemma \ref{lm:paths} follows from Claim \ref{cl1}; indeed, one can construct the claimed sequence recursively as follows: given the partially constructed sequence $v=v_0,v_1,\dots,v_s$ we append either the single term $a(v_s)$ or the two terms $z(v_s),a(v_s)$, according to which case of part \eqref{aorz} of the claim applies, and terminate if $\phi(a(v_s))=\ell_0$. Observe that a consecutive pair $v,v'$ in $v_0,\dots,v_t$ belongs (as an unordered pair) to $\bar J_i$ only if either
  \begin{enumerate}[(i)]
  \item $v'=a(v)$ and $\phi(v')>\phi(v)$, or
  \item $v'=z(v)$, the term after $v'$ is $v''=a(v)$, and $\psi(v'')>\psi(v)$, or
  \item the term before $v$ is $\hat v$, $v=z(\hat v)$, $v'=a(\hat v)$, and $\psi(v')>\psi(\hat v)$.
  \end{enumerate}
  Since $(\phi(v),\psi(v))$ increases lexicographically in this way along the path, we have the claimed upper bound of $1+2(\ell_0-i)$ on the number of of consecutive pairs from $\bar J_i$. This finishes the proof of Lemma \ref{lm:paths}.
\end{proof}

Now we are ready to finish the proof of Lemma \ref{extralemmaB}.  Thanks to Lemma \ref{lm:paths}, we can decompose $\eta_v - \eta_{\alpha(v)} = \sum_{j=1}^t \eta_{j-1} - \eta_j$ and using a chaining argument as before we get 
\begin{align}
\E_{H,\eta}\left(\max_{v \in V} |\eta_v - \eta_{\alpha(v)}|\right) &\leq \E_H\brac{\sum_{i=0}^{\ell_0} (1+2(\ell_0-i))\E_\eta \max_{\{a,b\} \in \bar J_i} |\eta_a - \eta_b|} \nonumber\\
&\leq O\brac{\E_H\brac{\sum_{i=0}^{\ell_0} (\ell_0-i+1)\sqrt{2^i}(\sqrt{2\log|J_i|})}}.\label{Jensen}
%&\leo \sum_{i=0}^{\ell_0} (\ell_0-i+1)\sqrt{2^i}\sqrt{2\log\left(\frac{\e n}{2^{2i}}\right)}.\nonumber
\end{align}
Here, $\E_{H,\eta}$ is expectation over the larger space of the random graph $H$ together with the GFF, while $\E_\eta$ is the expectation of a fixed Gaussian Free Field and $\E_H$ is an expectation just over the random choice of $H$ (this is to handle $\sqrt{\log |J_i|}$, as we do not have a high probability statement about $|J_i|$ covered by APOH and we will only be able to control $\E_H |J_i|$). The first inequality follows from part (b) of Lemma \ref{lm:paths} and the second inequality follows from the union bound on the maximum, \eqref{emax}.

Given \eqref{Jensen}, our task is to bound $\E_H(|J_i|)$ for $0\leq i\leq \ell_0$ and then show that the sum in \eqref{Jensen} is $o(T_\delta)$. We have from Property \ref{representative} that  
\beq{Jisize}{
\E_H (|J_i|) =O\brac{ \E_H|W_{2^i}|\times \frac{1}{2^i}} = O\brac{ (\e^2 n)\times\sum_{j\geq 0}\m^{j2^i}\times \frac{1}{2^i} } = O\brac{ \frac{\e^2n}{2^i(1-\m^i)}} = O\brac{ \frac{\e n}{2^{2i}} }
}
(the number of vertices on $K_2$ is $\e^2n$ and $\mu^{j2^i}$ bounds the expected number of vertices on level $j2^i$.) Going back to \eqref{Jensen} we see that 
\beq{only}{
\E_{H,\eta}\left(\max_{v \in V} |\eta_v - \eta_{\alpha(v)}|\right)  \leq \sum_{i=0}^{\ell_0} (\ell_0-i+1)\sqrt{2^i}\sqrt{2\log\left(\frac{\e n}{2^{2i}}\right)}.
}
Here we use that $\E_H(\sqrt{\log|J_i|}) \leq \sqrt{\log \E(|J_i|)}$, by Jensen's inequality ($\log^{1/2}x$ is a concave function) and \eqref{Jisize}. 

It only remains to deal with the R.H.S. of \eqref{only}. Given $v \in V$, we let $u(v)$ to be a closest vertex  in $U$ to $v$ (in the graph distance). Suppose for now that $u(v) = \alpha(v)$, where $\alpha(v)$ is provided by Lemma \ref{lm:paths}.

To get a high probability result, we will use the Markov inequality: if we denote $Y = \E_\eta\left(\max_{v \in V} |\eta_v - \eta_{\alpha(v)}|\right)$, we have $\Pr_H\brac{Y > (\log N)^{1/4}\E_H Y} \leq (\log N)^{-1/4}$ and this explains the $\log^{1/4}N$ factor in \eqref{lastone} below. We check that the ratio between the terms $i+1$ and $i$ in \eqref{only} equals
\[
\frac{\ell_0-i}{\ell_0-i+1}\sqrt{2}\sqrt{1 - \frac{2\log 2}{\log(\e n) - 2i\log 2}}
\]
which is strictly larger than, say $\frac{10}{9}$ for $0 \leq i \leq \ell_0-10$. Thus the last $10$ terms dominate this sum and we get that w.h.p.
\beq{lastone}{
\E_\eta (\max_{v \in V} |\eta_v - \eta_{\alpha(v)}|) \leq O\brac{\log^{1/4}N\times \sqrt{2^{\ell_0}}\sqrt{2\log\left(\frac{\e n}{2^{2\ell_0}}\right)}} = O\brac{\frac{\log^{3/4} N}{\e^{1/2}}} = o(T_\delta).
}

This concludes the proof of Lemma \ref{extralemmaB} in the case $u(v) = \alpha(v)$. If $u(v) \neq \alpha(v)$, then since $\eta_v - \eta_{u(v)} = (\eta_v - \eta_{\alpha(v)}) + (\eta_{\alpha(v)} - \eta_{\alpha(\alpha(v))}) + (\eta_{\alpha(u(v))} - \eta_{u(v)})$, by the triangle inequality we can obtain the same bound as above up to the constant $3$.


\begin{thebibliography}{99}
%\bibitem{ACF}

\bibitem{AKLLR} R. Aleliunas, R.M. Karp, R.J. Lipton, L. Lov\'asz and C. Rackoff, Random Walks, Universal Traversal Sequences, and the Complexity of Maze Problems. {\em Proceedings of the 20th Annual IEEE Symposium on Foundations of Computer Science} (1979) 218-223.

\bibitem{BDNP} M. Barlow, J. Ding, A. Nachmias and Y. Peres, The evolution of the cover time, The evolution of the cover time, {\em Combinatorics, Probability and Computing} 20 (2011) 331-345.

\bibitem{BM} B. Bollob\'as, A probabilistic proof of an asymptotic formula for the number of labelled graphs, {\em European Journal on Combinatorics} 1(1980) 311-316.

\bibitem{Chatterjee} S. Chatterjee, Superconcentration and related topics, Springer, 2014.

\bibitem{CF1} C. Cooper and A. M. Frieze, The cover time of sparse random graphs, {\em Random Structures and Algorithms} 30 (2007) 1-16.

\bibitem{CF2} C. Cooper and A. M. Frieze, The cover time of random regular graphs, {\em SIAM Journal on Discrete Mathematics}, 18 (2005) 728-740.

\bibitem{CF3} C. Cooper and A. M. Frieze, The cover time of random geometric graphs, {\em Random Structures and Algorithms} 38 (2011) 324-349.

\bibitem{CF4} C. Cooper and A. M. Frieze, Stationary distribution and cover time of random walks on random digraphs, {\em Journal of Combinatorial Theory} B 102 (2012) 329-362.

\bibitem{CFgiant} C. Cooper and A. M. Frieze, The cover time of the giant component of a random graph, {\em Random Structures and Algorithms}, 32, 401-439 (2008).

\bibitem{CFL} C. Cooper, A.M. Frieze and E. Lubetzky, Cover time of a random graph with given degree sequence II: Allowing vertices of degree two, {\em Random Structures and Algorithms} 45 (2014) 627-674.

\bibitem{Cram} H. Cram\'er, Mathematical Methods of Statistics, Princeton Mathematical Series, Vol. 9. Princeton University Press, Princeton, N. J., 1946. (Formula (28.6.16) on p. 376.)

\bibitem{DKLP} J. Ding, J. Kim, E. Lubetzky and Y. Peres, Anatomy of a young giant component in the random graph, {\em Random Structures and Algorithms} 39 (2011) 139-178.

\bibitem{DKLP2} J. Ding, J. Kim, E. Lubetzky and Y. Peres, Diameters in supercritical random graphs via first passage percolation,
{\em Combin. Probab. Comput.} 19 (2010), no. 5-6, 729-751.

\bibitem{DLP} J. Ding, J.R. Lee and Y. Peres, Cover times, blanket times, and majorizing measures, {\em Annals of Mathematics} 175 (2012) 1409-1471.

\bibitem{Ding} J. Ding, Asymptotic of cover times via Gaussian free fields: Bounded degree graphs and general trees, {\em  The Annals of Probability} 42 (2014) 464-496.

\bibitem{DS} P. Doyle and J. Snell, Random walks and electric networks, The Mathematical Association of America, 2006, (https://math.dartmouth.edu/~doyle/docs/walks/walks.pdf).

\bibitem{Feige1} U. Feige, A tight upper bound for the cover time of random walks on graphs, {\em Random Structures and Algorithms}, 6 (1995) 51-54.

\bibitem{Feige2} U. Feige, A tight lower bound for the cover time of random walks on graphs, {\em Random Structures and Algorithms}, 6 (1995) 433-438.

\bibitem{Fern} X. Fernique, Regularit\'e des trajectoires des fonctions aleatoires gaussiennes, \'Ecole d'Ete de Probabilites de St-Flour 1974, Lecture Notes in Mathematics, vol. 480, Springer, Berlin Heidelberg 1975, pp. 1-96.

\bibitem{Led} M. Ledoux, The concentration of measure phenomenom, Mathematical Surveys and Monograph 89, 2001.

\bibitem{MQS} F. Manzo, M. Quattropani and E. Scoppola, \href{https://arxiv.org/pdf/2103.04667.pdf}{A probabilistic proof of Cooper and Frieze's "First Visit Time Lemma"}.

\bibitem{LPW} D. Levin, Y. Peres and E. Wilmer, Markov Chains and Mixing Times: Second Edition, American Mathematical Society, 2017.

\bibitem{Tal} M. Talagrand, Upper and lower bounds for stochastic processes. Modern methods and classical problems, Springer, Heidelberg, 2014.

\bibitem{Z} A. Zhai, Exponential concentration of cover times, {\em Electronic Journal of Probability} 23 (2018).
\end{thebibliography}
\end{document}